\RequirePackage{fix-cm}
\documentclass[smallextended]{svjour3}       
\smartqed  
\usepackage{amssymb, amsmath, latexsym}
\usepackage[colorlinks=true,linkcolor=blue,citecolor=blue]{hyperref}
\usepackage{amsmath, amscd, amsfonts, amssymb, graphicx, color}
\usepackage{epsfig, epstopdf}
\usepackage{cite}
\usepackage{float}
\begin{document}

\title{ Two generalized strong convergence algorithms for the variational
inequality problems in Banach spaces
}

\titlerunning{Two generalized strong convergence algorithms}        

\author{Mostafa Ghadampour$^{1}$          
    \and
        Ebrahim Soori$^{*,2}$ 
}


\institute{$^*$Corresponding author\\
Mostafa Ghadampour \at
               Department  of Mathematics, Lorestan University, Lorestan, Khoramabad, Iran. \\
              \email{m.ghadampour@gmail.com}           
           \and
           Donal O$'$Regan \at
            School of Mathematics, Statistics, National University of Ireland, Galway, Ireland.\\
           \email{donal.oregan@nuigalway.ie}
           \and
          Ebrahim Soori \at
                Department  of Mathematics, Lorestan University, Lorestan, Khoramabad, Iran.\\
               \email{sori.e@lu.ac.ir}
               \and
              Ravi P. Agarwal\at
                Department of Mathematics Texas A$\&$M University-Kingsville 700 University Blvd., MSC 172 Kingsville, Texas, USA.\\
 \email{agarwal@tamuk.edu}
}

\date{Received: date / Accepted: date}

\maketitle

\begin{abstract}
In this paper, two generalized algorithms for solving the variational inequality problem in Banach spaces are proposed. Then the strong convergence of the sequences generated by these algorithms will be proved under the suitable conditions. Finally, using MATLAB software,
 we provide some numerical examples to illustrate our results.
\end{abstract}
\section{Introduction}
Let $C$ be a nonempty closed convex subset of a Banach space $E$ with norm $\|.\|$ and let $E^*$ denotes the dual of $E$. The variational
inequality problem (VIP) is to find a point $x\in C$ such that
       \begin{equation}\label{vip}
         \langle Ax, y - x\rangle \geq 0 \;\;\;\forall\;\; y \in C,
       \end{equation}
where $A$ is a mapping of $C$ into $E^*$ and $\langle.,.\rangle$   denotes the pairing between $E$ and $E^*$.
The solutions set of \eqref{vip} is denoted by $VI(A,C)$.
\\
It is well known that variational inequalities cover a variety of fields in optimal control, optimization, mathematical programming, operational research,
partial differential equations, engineering, and equilibrium models and hence, it have been studied by many authors in the recent years\cite{jlotaa, tdvh, vpts, cgga}.

The operator $A$ of $C$ to $E^*$ is said to be\\
(i) monotone if
    \begin{equation*}
      \langle x - y, Ax - Ay\rangle \geq 0, \;\; \forall x, y \in C;
    \end{equation*}
(ii) $\alpha-$inverse strongly monotone if there exists a constant $\alpha > 0$ such that
     \begin{equation*}
       \langle x - y, Ax - Ay\rangle\geq \alpha\|Ax -Ay\|^2 \;\; \forall x, y \in C;
     \end{equation*}
(iii) $L$-Lipchitz continuous if there exists $L>0$ such that
     \begin{align*}
        \| Ax-Ay\|\leq L\| x-y\|,\;\;\forall x,y\in C.
     \end{align*}

Let $f :C \times C\rightarrow \mathbb{R}$ be a bifunction. The equilibrium problem (GEP) is as follows:  Find $x \in  C $ such that
\begin{equation}\label{GEP}
      f(x, y) + \langle Ax, y - x\rangle \geq 0, \;\; \forall\; y\in C.
  \end{equation}
The set of solutions of \eqref{GEP} is denoted by $GEP(f, A)$. Clearly, the problem \eqref{GEP} is equivalent to (VIP) if $f \equiv 0$.

Korpelevich\cite{ko} proposed the following algorithm for solving the problem (VIP) that is known as extragradient method as \eqref{kor}. Let $x_1$ be an arbitrarily element in $H$:
    \begin{equation}\label{kor}
             \begin{cases}
                y_n=P_C(x_n-\lambda Ax_n), \\
                x_{n+1}=P_C(x_n-\lambda Ay_n),
             \end{cases}
          \end{equation}

Tseng \cite{tp} proposed the following algorithm which was introduced using the modified front-to-back (F-B) method.
          \begin{equation}\label{algo1}
            \begin{cases}
              y_n=P_C(x_n-\lambda Ax_n),  \\
              x_{n+1}=P_X(y_n-\lambda(Ay_n-Ax_n)),
           \end{cases}
         \end{equation}
where $X = C$ and $X=H$ if $A$ is Lipschitz continuous. Thong et al \cite{th} proposed the following convergent algorithm based on the Tseng algorithm.
          \begin{equation}\label{algo}
              \begin{cases}
                y_n=P_C(x_n-\lambda_n Ax_n), \\
                z_n= y_n - \lambda_n(Ay_n -Ax_n),\\
                x_{n+1}=\alpha_nf(x_n)+(1-\alpha_n) z_n,
              \end{cases}
           \end{equation}
where the operator $A$ is monotone and Lipschitz continuous, $\gamma > 0, \; l\in (0,1),  \;\mu \in (0, 1)$ and $\lambda_n$ is chosen to be the largest $\lambda
\in \{ \gamma, \gamma l, \gamma l^2, ...\}$ satisfying
           \begin{align}\label{1}
             \lambda\| Ax_n - Ay_n \| \leq \mu \| x_n - y_n \|.
           \end{align}
In this paper, we present our algorithms in Banach spaces motivated by the Thong algorithm and prove the strong convergence of the sequences generated by these algorithms.
 Finally, using MATLAB software, we provide some numerical examples to illustrate our claims.
\section{Preliminaries}
\label{sec:1}
Let $E$ be a real Banach space with norm $\| .\|$ and let $E^*$ be the dual space of $E$.
 The strong convergence and the weak convergence  of the sequence $\{x_n\}$ to $x$ in $E$ are denoted by $x_n\rightarrow x$ and $x_n\rightharpoonup x$ through in the paper, respectively. The modulus $\delta$ of convexity of $E$ is defined by
$$\delta(\epsilon)=\inf\{1-\frac{\| x+y\|}{2}:\| x\|\leq1,\| y\|\leq1,\| x-y\|\geq\epsilon\}$$
for every $\epsilon \in [0, 2]$. $A$ Banach space $E$ is said to be uniformly convex if $\delta(0)=0$ and $\delta(\epsilon)>0$ for every $\epsilon> 0 $. It is
known that a Banach space $E$ is uniformly convex if and only if for any two sequences $\{x_n\}$ and $\{y_n\}$ in $E$ such that
   \begin{equation*}
     \lim_{n\rightarrow\infty}\| x_n\|=\lim_{n\rightarrow \infty}\| y_n\|=1\; and \;\lim_{n\rightarrow \infty}\| x_n + y_n\|=2,
   \end{equation*}
$\lim_{n\rightarrow \infty}\| x_n - y_n \|=0$ holds. Suppose that $p$ is a fixed real number with $p\geq 2$. A Banach space $E$ is said to be $p$-uniformly
convex\cite{tyhk}, if there exists a constant $c>0$ such that $\delta\geq c\epsilon^p$ for all $\epsilon\in [0,2]$. It is also known that a uniformly convex Banach
space has the Kadec-Klee property, that is, $x_n\rightharpoonup u$ and $\| x_n \|\rightarrow\| u\|$ imply that $x_n \rightarrow u$(see \cite{cii, res}).
\\
 The normalized duality mapping $J: E \rightarrow E^*$ is defined by
     \begin{align*}
        J(x)=\{f \in E^*: \langle x, f \rangle= \|x\|^2=\|f\|^2 \},
     \end{align*}
for each $x \in E$.  Suppose that   $S(E) = \{x \in E : \|x\| = 1\}$.   A Banach space $E$ is called  smooth if for all $x \in S(E)$,
there exists a unique functional $j_x \in E^*$ such that $\langle x, j_x\rangle = \|x\|$ and $\|j_x\| = 1$( see \cite{Ag}).
\\

The norm of $E$ is said to be $G\hat{a}teaux$
differentiable if for each $x,y\in S(E)$, the limit\\
   \begin{equation}\label{smo}
     \lim_{t\rightarrow 0}\frac{\| x+ty\| - \| x \|}{t}
   \end{equation}
exists. In this case, $E$ is called smooth and $E$ is said to be uniformly smooth if the limit \eqref{smo} is attained uniformly for all $x,y\in S(E)$\cite{tn}.
If a Banach space $E$ is uniformly convex, then $E$ is reflexive and strictly convex, and $E^*$ is uniformly smooth\cite{Ag}. It is well known that if $E$ is a
reflexive, strictly convex and smooth Banach space and $J^*:E^*\rightarrow E$ be the duality mapping on $E^*$, then $J^{-1}=J^*$,
also, if $E$ is a uniformly smooth Banach space, then $J$ is uniformly norm to norm continuous on bounded sets of $E$ and
$J^{-1}=J^*$ is also uniformly norm to norm continuous on bounded sets of $E^*$.
Let $E$ be a smooth Banach space and let $J$ be the duality mapping on $E$. The function $\phi:E\times E\rightarrow \mathbb{R}$ is define by
   \begin{equation}\label{phi}
     \phi(x,y)=\| x\|^2-2\langle x,Jy\rangle+\| y\|^2,\;\;\;\forall x,y\in E.
   \end{equation}
Clearly, from \eqref{phi}, it is concluded that
   \begin{equation}\label{phi1}
     (\| x\|-\| y\|)^2\leq\phi(x,y)\leq(\| x\|+\| y\|)^2.
   \end{equation}
If $E$ is a reflexive, strictly convex and smooth Banach space, then for all $x, y \in E$
   \begin{equation}\label{phi4}
     \phi(x,y)=0  \Leftrightarrow  x=y.
   \end{equation}
Also, It is obvious from the definition of the function $\phi$ that the following conditions hold for all  $x, y, z, w \in E$,

   \begin{equation}\label{phi2}
     \phi(x,y)=\phi(x,z)+\phi(z,y)+2\langle x-z,Jz-Jy\rangle,
   \end{equation}
   \begin{equation}\label{phi3}
     2\langle x-y,Jz-Jw\rangle=\phi(x,w)+\phi(y,z)-\phi(x,z)-\phi(y,w).
   \end{equation}
   \begin{equation}\label{phixy}
        \phi(x, y)=\langle x, Jx- Jy\rangle +\langle y-x, Jy\rangle\leq\| x\|\| Jx-Jy\|+\| y-x\|\| y\|.
   \end{equation}
Now, the function $V : E \times E^*\rightarrow\mathbb{R}$ is defined as follows
$$V(x, x^*)= \| x\|^2-2 \langle x, x^*\rangle+\| x^*\|^2,$$
for all $x \in E$ and $x^*\in E$. Moreover, $V(x, x^*) = \phi(x, J^{-1}x^*)$ for all $x\in E$ and $\in E$.
If $E$ is a reflexive strictly convex and smooth Banach space with $E^*$ as its dual, it is concluded that
   \begin{equation}\label{vq}
     V(x,x^*)+2\langle J^{-1}x^*-x,y^*\rangle\leq V(x,x^*+y^*),
   \end{equation}
for all $x\in E$ and all $x^*,y^*\in E^*$\cite{kt}.

 An operator $A : C\rightarrow E^*$ is hemicontinuous at $x_{0}\in C$, if for any sequence $\{x_n\}$ converging to $x_{0}$
along a line implies that $Tx_n\rightharpoonup Tx_{0}$, i.e., $Tx_n = T(x_{0} + t_nx)\rightharpoonup Tx_{0}$ as $t_n\rightarrow 0$ for all $x\in C$.\\
The generalized projection $\Pi_{C} : E\rightarrow C$ is a mapping that assigns to an arbitrary
point $x\in E$, the minimum point of the functional $\phi(y, x)$; that is, $\Pi_{C}x =x_{0}$, where
$x_{0}$ is the solution of the minimization problem
   \begin{equation}\label{min}
     \phi(x_{0},x)=\min_{y \in C}\phi(y,x).
   \end{equation}
The existence and uniqueness of the operator $\Pi_{C}$ follows from the properties of the
functional $\phi(x, y)$ and strict monotonicity of the mapping $J$\cite{ayl}.
Suppose that $C$ is a nonempty closed convex subset of $E$, and $T$ is a mapping from $C$
into itself. A point $p\in C$ is called an asymptotically fixed point of $T$ if $C$
contains a sequence $\{x_n\}$ which converges weakly to $p$ such that $Tx_n - x_n\rightarrow 0$\cite{Ag}.
The set of asymptotical fixed points of $T$ will be denoted by $\hat{F}(T)$. A mapping
$T$ from $C$ into itself is said to be relatively nonexpansive  if $\hat{F}(T) = F(T)$ and $\phi(p,Tx)\leq \phi(p,x)$ for all $x\in C$ and $p\in F(T)$. The
asymptotic behavior of a relatively nonexpansive mapping was studied in \cite{bdrs, bdrs1, cyrs}.\\
We need the following lemmas for the proof of our main results.
   \begin{lemma}\label{2.1.0}
   (\cite{ktw}) Let $E$ be a smooth and uniformly convex Banach
     space and let $\{x_n\}$ and $\{y_n\}$ be two sequences of $E$. If $\phi(x_n,y_n)\rightarrow 0$ and either $\{x_n\}$ or $\{y_n\}$ is bounded,
     then $x_n- y_n\rightarrow 0$.
   \end{lemma}
   \begin{lemma}\label{2.1.1}
(\cite{ayl}) Let $C$ be a nonempty closed convex subset of a
      smooth, strictly convex and reflexive Banach space $E$, let $x\in E$ and let $z\in C$.
      Then\\
\hspace*{3cm}        $z = \Pi_{C}x\Leftrightarrow \langle y - z, Jx - Jz\rangle \leq 0$, for all $y\in C$.

   \end{lemma}
   \begin{lemma}\label{2.1.2}
     (\cite{ayl}) Let $C$ be a nonempty closed convex subset of a
      smooth, strictly convex and reflexive Banach space $E$ and let $y\in E$. Then\\
     \hspace*{2cm}  $\phi(x,\Pi_{C}y) +\phi(\Pi_{C}y, y)\leq \phi(x, y),\;\; \forall x \in C$.
   \end{lemma}
   \begin{lemma}\label{2.2}
(\cite{bb, xuh}) Let $E$ be a 2-uniformly convex and smooth Banach space. Then, for all $x$, $y\in E$,
      it is concluded that\\
      \hspace*{4cm}$\| x - y\| \leq \frac{2}{c^2}\| Jx-Jy\| $,\\
      where $\frac{1}{c}(0\leq c\leq1) $is the 2-uniformly convex constant of $E$.
   \end{lemma}
   \begin{lemma}\label{2.4}
     (Xu \cite{xuh}). Let $E$ be a uniformly convex Banach space and $r>0$. Then there exists a
      continuous strictly increasing convex function $g : [0, 2r] \rightarrow [0,\infty)$ such that $g(0) = 0$ and
      \begin{align*}
      \| tx + (1 - t)y\|^2\leq t\| x\|^2 + (1 - t)\| y\|^2 - t(1 - t)g(\| x - y\|),
      \end{align*}
for all $x, y \in B_{r}(0) = \{z \in E : \| z\|\leq r\}$ and  $t\in [0, 1]$.
   \end{lemma}
   \begin{lemma}\label{2.5}
     (\cite{ktw}). Let $E$ be a uniformly convex Banach space and $r>0$. Then there exists a
      continuous strictly increasing convex function $g : [0, 2r] \rightarrow [0,\infty)$ such that $g(0) = 0$ and
      \begin{align*}
        g(\| x - y\|)\leq \phi (x, y),
      \end{align*}
for all $x, y \in B_{r}(0) = \{z \in E : \| z\|\leq r\}$.
     \end{lemma}

Throughout this paper, we assume that $f : C \times C \rightarrow \mathbb{R}$ be a bifunction satisfying the following conditions
     \begin{enumerate}
       \item [(A1)] $f(x,x)=0$  for all $x\in C$,
       \item [(A2)] f is monotone, i.e. $f(x, y)+ f(y, x) \leq 0$, for all $x, y \in C$,
       \item [(A3)] $\displaystyle\lim_{t\downarrow 0}f(tz+(1-t)x, y)\leq f(x, y)$, for all $x,y,z \in C$,
       \item [(A4)] for each $x\in C, y\mapsto f(x, y)$ is convex and lower semicontinuous.
     \end{enumerate}
   \begin{lemma}\label{2.7}
(\cite{lyc}) Let $C$ be a nonempty closed convex subset of a smooth, strictly convex and
reflexive Banach space $E$. Let $A :C\longrightarrow E^*$ be an $\alpha-$inverse-strongly monotone operator and $f$ be a bifunction from $C \times C$ to
$\mathbb{R}$ satisfying $(A_{1}) - (A_{4})$. Then for all $r > 0$ hold the following
    \begin{enumerate}
      \item [(i)] for $x\in E$, there exists $u\in C$ such that
           \begin{equation*}
              f(u,x)+ \langle Au,y-u \rangle+\frac{1}{r}\langle y-u,Ju-Jx\rangle\geq 0,\;\;\;\forall y\in C,
           \end{equation*}
      \item [(ii)]  if $E$ is additionally uniformly smooth and $K_{r}: E\longrightarrow C$ is defined as
     \begin{equation*}\label{kr}
       K_{r}(x)=\{ u\in C\;\; :\;\; f(u,y)+\langle Au,y-u\rangle+\frac{1}{r}\langle y-u,Ju-Jx\rangle\geq 0,\;\;\;\forall y\in C \}
     \end{equation*}
    \end{enumerate}
Then, the following conditions hold:
    \begin{enumerate}
      \item [(1)] $K_{r}$ is single-valued,
      \item [(2)] $K_{r}$ is firmly nonexpansive, i.e., for all $x, y\in E$,
         \begin{equation*}
           \langle K_rx - K_ry, JK_rx - JK_ry\rangle \leq \langle K_rx - K_ry, Jx - Jy\rangle,
         \end{equation*}
      \item [(3)] $F(K_{r}) = \hat{F(K_{r})}= GEP(f , A)$,
      \item [(4)] $GEP$ is a closed convex subset of $C$,
      \item [(5)] $\phi(p,K_{r}x)+\phi(K_{r}x,x)\leq \phi(p,x), \;\;\forall\;\;p\in F(K_{r})$.
    \end{enumerate}
   \end{lemma}
The normal cone for $C$ at a point $\upsilon\in C$ is denoted by $N_{C}(\upsilon)$, that is
$N_{C}(\upsilon) := \{x^*\in E^* : \langle \upsilon - y, x^*\rangle \geq 0, \forall y\in C\}$.
   \begin{lemma}\label{2.60}
(\cite{rrt}) Let $C$ be a nonempty closed convex subset of a Banach space $E$ and let $T$ be
monotone and hemicontinuous operator of $C$ into $E^*$ with $C = D(T)$. Let $B\subset E \times E^*$ be an
operator define as follows:\\
\begin{equation*}
 B v =\left\{
\begin{array}{lr}
Tv + N_{C}v,\qquad v \in C, \\
\emptyset,\qquad \qquad\qquad v \notin C.
\end{array} \right.
\end{equation*}
Then $B$ is maximal monotone and $B^{-1}(0) =$ SOL$(T,C)$.
   \end{lemma}
 \section{Main results}
In this section, we introduce a new iterative algorithms for solving monotone variational
inequality problems which are based on Tseng’s intergradient method.We prove strong convergence theorems for generated sequences by presented intergradient
algorithms, under suitable conditions.

 Throughout this section, we assume that $C$ is a nonempty closed convex subset of a real 2-uniformly convex and uniformly smooth Banach
space $E$ and $E^*$ is the dual space of $E$, $A : C\rightarrow E^*$ is an $\alpha$-inverse strongly monotone operator. Assume that $\{\lambda_n\}$ is a sequence of real numbers such that $0<\lambda_n < \frac{c^2\alpha}{2}$ for all $n\in \mathbb{N}$, where $\frac{1}{c}$ is the 2-uniformly convexity constant of $E$.
\begin{theorem}\label{jmen}  Let $x_{0}\in C$, $\Gamma:=VI(C, A)\cap F(f)\neq \emptyset$ and
   \begin{equation}\label{algo}
     \begin{cases}
       y_n=\Pi_CJ^{-1}(Jx_n-\lambda_nAx_n),\\
       z_n=J^{-1}(Jy_n-\lambda_nAy_n),\\
       x_{n+1}=\Pi_CJ^{-1}(\alpha_{n,1}Jx_n +\alpha_{n,2}Jf(x_n)+\alpha_{n,3}Jz_n),
    \end{cases}
   \end{equation}
where $\{\lambda_n\}\subseteq [0, 1]$ such that $\displaystyle\lim_{n\rightarrow \infty}\lambda_n=0$. Let $\{\alpha_{n,i}\}\subset(0,1)$ for $i=1, 2, 3$, $\alpha_{n,1}+\alpha_{n,2}+\alpha_{n,3}=1$ and $\displaystyle\liminf_{n\rightarrow\infty}\alpha_{n,2}\alpha_{n,3}>0$. Let f be a relatively nonexpansive
    self-mapping on $C$ and $\| Ax\|\leq \| Ax-Au\|$ for all $x\in C$ and $u\in \Gamma$. Consider the sequence $\{x_n\}$
generated by the algorithm \eqref{algo}. Then the sequence $\{x_n\}$ converges strongly to $q=\Pi_{VI(C,A)}\circ f(q)$, where $P_{VI(C,A)}\circ f: H \rightarrow VI(C, A)$ is the mapping defined by $P_{VI(C,A)}\circ f(x)=P_{VI(C,A)}(f(x))$ for each $x\in H$.
  \end{theorem}
\begin{proof}
Let $ \hat{u}\in \Gamma$. From the definition of function $V$ and the inequality \eqref{vq}, it is concluded that
       \begin{align}\label{uzn}
             \phi (\hat{u},z_n)=& \phi(\hat{u},J^{-1}(Jy_n-\lambda_nAy_n))\nonumber \\
                =&V(\hat{u},Jy_n-\lambda_nAy_n)\nonumber \\
             \leq& V(\hat{u},Jy_n)-2\langle J^{-1}(Jy_n-\lambda_nAy_n)-\hat{u},\lambda_nAy_n\rangle\nonumber\\
             =&\phi (\hat{u},y_n)+2\langle J^{-1}(Jy_n-\lambda_nAy_n)-J^{-1}(Jy_n),-\lambda_nAy_n\rangle\nonumber\\
             &-2\langle y_n-\hat{u},\lambda_nAy_n\rangle,
       \end{align}
then from Lemma \ref{2.2} and the condition $\| Ax\| \leq \| Ax-A\hat{u} \|$ for all $x\in C$, it is followed that
        \begin{align}\label{ineq1}
            2\langle J^{-1}(Jy_n- & \lambda_nAy_n)-J^{-1}(Jy_n),-\lambda_nAy_n\rangle\nonumber \\
           \leq & 2\| J^{-1}(Jy_n-\lambda_nAy_n)-J^{-1}(Jy_n)\|\|
           -\lambda_nAy_n\| \nonumber\\
           \leq & \frac{4\lambda_n^2}{c^2}\| Ay_n\|^2\nonumber\\
           \leq & \frac{4\lambda_n^2}{c^2}\| Ay_n-A\hat{u}\|^2.
         \end{align}
Since $A$ is $\alpha$-inverse strongly monotone and the fact that $\hat{u}\in VI(C, A)$, we have
       \begin{align}\label{ineq2}
          -2\langle y_n-\hat{u},&\lambda_nAy_n\rangle\nonumber\\
             =&-2\lambda_n\langle y_n-\hat{u}, Ay_n-A\hat{u}\rangle-2\lambda_n\langle y_n-\hat{u}, A\hat{u}\rangle\nonumber\\
             \leq&-2\lambda_n\langle y_n-\hat{u}, Ay_n-A\hat{u}\rangle\nonumber\\
             \leq &-2\lambda_n\alpha\| Ay_n-A\hat{u}\|^2,
       \end{align}
substituting \eqref{ineq1} and \eqref{ineq2} in \eqref{uzn} and using our assumptions, we obtain
       \begin{align*}
             \phi(\hat{u},z_n)\leq& \phi(\hat{u},y_n)+(\frac{4\lambda_n^2}{c^2}-2\lambda_n\alpha)\|  Ay_n-Au\|^2\\
             =&\phi(\hat{u},y_n)+2\lambda_n(\frac{2\lambda_n}{c^2}-\alpha) \|  Ay_n-A\hat{u}\|^2\\\leq& \phi(\hat{u},y_n),
       \end{align*}
hence,
       \begin{equation}\label{one}
          \phi(\hat{u}, z_n)\leq \phi(\hat{u},y_n).
       \end{equation}
From Lemma \ref{2.1.2} and the inequality \eqref{vq}, we have
       \begin{align}\label{xxx}
           \phi(\hat{u},y_n)=& \phi(\hat{u},\Pi_CJ^{-1}(Jx_n-\lambda_nAx_n))\nonumber\\
           \leq& \phi(\hat{u},J^{-1}(Jx_n-\lambda_nAx_n))=V(\hat{u},Jx_n-\lambda_nAx_n)\nonumber\\
           \leq& V(\hat{u},Jx_n)-2\langle J^{-1}(Jx_n-\lambda_nAx_n)-\hat{u},\lambda_nAx_n\rangle\nonumber\\
           =& \phi(\hat{u},x_n)-2\lambda_n\langle x_n-\hat{u},Ax_n\rangle\nonumber\\
           &+2\langle J^{-1}(Jx_n-\lambda_nAx_n)-J^{-1}(Jx_n),-\lambda_nAx_n\rangle,
       \end{align}
since $A$ is $\alpha-$inverse strongly monotone and $\hat{u}\in VI(C, A)$, it follows that
       \begin{align}\label{ary}
           -2\lambda_n\langle x_n-\hat{u},&Ax_n\rangle\nonumber\\
           =&-2\lambda_n\langle x_n-\hat{u},Ax_n-A\hat{u}\rangle-2\lambda_n\langle x_n-\hat{u},A\hat{u}\rangle\nonumber\\
           \leq&-2\lambda_n\langle x_n-\hat{u},Ax_n-A\hat{u}\rangle\nonumber\\
           \leq&-2\lambda_n\alpha \| Ax_n-A\hat{u}\|^2.
       \end{align}
From Lemma \ref{2.2} and our assumptions, we can conclude
    \begin{align}\label{t}
        2\langle J^{-1}(Jx_n-\lambda_n&Ax_n)-J^{-1}(Jx_n),-\lambda_nAx_n\rangle\nonumber\\
        \leq& 2\| J^{-1}(Jx_n-\lambda_nAx_n)- J^{-1}(Jx_n)\| \|-\lambda_nAx_n\|\nonumber\\
        \leq& \frac{4\lambda_n^2}{c^2}\| Ax_n\|^2\nonumber\\
        \leq& \frac{4\lambda_n^2}{c^2}\| Ax_n-A\hat{u}\|^2.
     \end{align}
By applying \eqref{ary} and \eqref{t} in \eqref{xxx} and our assumptions, it is implied that
   \begin{equation}\label{two}
      \phi(\hat{u},y_n)\leq \phi(\hat{u},x_n)+2\lambda_n(\frac{2\lambda_n}{c^2}-\alpha) \|  Ay_n-Ax_n\|^2\leq \phi(\hat{u},x_n).
   \end{equation}
Hence, from \eqref{one} and \eqref{two}, we have
\begin{equation}\label{three}
 \phi(\hat{u},z_n)\leq \phi(\hat{u},x_n).
\end{equation}
Next, it will be shown that the sequence $\{\phi(\hat{u},x_n)\}$ is decreasing.
From the relatively  nonexpansiveness condition of $f$, convexity of $\|.\|^2$, Lemma \ref{2.1.2} and the inequality \eqref{three}, it is implied that
   \begin{align}\label{decreas}
      \phi(\hat{u}, x_{n+1})\leq& \phi(\hat{u}, J^{-1}(\alpha_{n,1}Jx_n +\alpha_{n,2}Jf(x_n)+\alpha_{n,3}Jz_n)\nonumber\\
      =&\| \hat{u}\|^2-2\langle \hat{u}, \alpha_{n,1}Jx_n +\alpha_{n,2}Jf(x_n)+\alpha_{n,3}Jz_n\rangle \nonumber\\
      &+\| \alpha_{n,1}Jx_n +\alpha_{n,2}Jf(x_n)+\alpha_{n,3}Jz_n\|^2\nonumber\\
      \leq& \| \hat{u}\|^2-2\alpha_{n,1}\langle \hat{u}, Jx_n\rangle-2\alpha_{n,2}\langle \hat{u}, Jf(x_n)\rangle-2\alpha_{n,3}\langle \hat{u}, Jz_n\rangle\nonumber \\
      &+\alpha_{n,1}\|x_n\|^2+\alpha_{n,2}\| f(x_n)\|^2+\alpha_{n,3}\| z_n\|^2\nonumber\\
      =& \alpha_{n,1}\phi(\hat{u},x_n)+\alpha_{n,2}\phi(\hat{u},f(x_n))+\alpha_{n,3}\phi(\hat{u}, z_n)\nonumber\\
      \leq & \alpha_{n,1}\phi(\hat{u},x_n)+\alpha_{n,2}\phi(\hat{u},x_n)+\alpha_{n,3}\phi(\hat{u}, x_n)\nonumber\\
      =& \phi(\hat{u}, x_n),
   \end{align}
so $\{\phi(\hat{u},x_n)\}$ is decreasing. Then it is implied that $\{\phi(\hat{u},x_n)\}$ is bounded, hence $\displaystyle \lim_{n\rightarrow\infty} \phi(\hat{u}, x_n)$ exists. Then from \eqref{phi1}, $\{x_n\}$ is bounded. It follows from the relatively nonexpansiveness condition of $f$, \eqref{two} and \eqref{three} that $\{f(x_n)\}$, $\{y_n\}$ and  $\{z_n\}$ are bounded. From Lemmas \ref{2.1.2}, \ref{2.2}, the inequality \eqref{vq} and the condition $\displaystyle\lim_{n\rightarrow\infty}\lambda_n=0$, we have
\begin{align}\label{a}
  \phi(x_n, y_n)&\leq \phi(x_n, J^{-1}(Jx_n-\lambda_nAx_n))\nonumber\\
  =&V(x_n, Jx_n-\lambda_nAx_n)\nonumber\\
  \leq&V(x_n, Jx_n)-2\langle  J^{-1}(Jx_n-\lambda_nAx_n)-x_n, \lambda_nAx_n)\nonumber\\
  =&\phi(x_n, x_n)-2\langle  J^{-1}(Jx_n-\lambda_nAx_n)-J^{-1}(Jx_n), \lambda_nAx_n)\rangle\nonumber\\
  \leq& 2\|J^{-1}(Jx_n-\lambda_nAx_n)-J^{-1}(Jx_n) \|\|\lambda_nAx_n \|\nonumber\\
   \leq&\frac{4\lambda_n^2}{c^2}\|Ax_n \|^2 \rightarrow 0 \;\;\;as\; n\rightarrow \infty.
 \end{align}
  By Lemma \ref{2.1.0}, it is implied that
 \begin{equation}\label{c}
   \displaystyle\lim_{n\rightarrow\infty}\|x_n-y_n\|= 0.
 \end{equation}
 Next, from \eqref{phixy}, \eqref{c}, the boundedness  of the sequences $\{x_n\}$ and $\{y_n\}$, and using uniformly norm-to-norm continuity of $J$ on bounded sets,
  it is obvious that
     \begin{equation}\label{c1}
      \phi(y_n,x_n)\leq\|y_n\|\|Jy_n-Jx_n\|+\|x_n-y_n\|\|x_n\|\rightarrow 0 \;\;\;as\; n\rightarrow \infty.
   \end{equation}
 By Lemmas \ref{2.1.2}, \ref{2.2}, the inequality \eqref{vq} and the condition $\displaystyle\lim_{n\rightarrow\infty}\lambda_n=0$, we have
  \begin{align}\label{d}
  \phi(y_n, z_n)&= \phi(y_n, J^{-1}(Jy_n-\lambda_nAy_n))\nonumber\\
  =&V(y_n, Jy_n-\lambda_nAy_n)\nonumber\\
  \leq&V(y_n, Jy_n)-2\langle  J^{-1}(Jy_n-\lambda_nAy_n)-y_n, \lambda_nAy_n)\nonumber\\
  =&\phi(y_n, y_n)-2\langle  J^{-1}(Jy_n-\lambda_nAy_n)-J^{-1}(Jy_n), \lambda_nAy_n)\rangle\nonumber\\
  \leq& 2\|J^{-1}(Jy_n-\lambda_nAy_n)-J^{-1}(Jy_n) \|\|\lambda_nAy_n \|\nonumber\\
   \leq&\frac{4\lambda_n^2}{c^2}\|Ay_n \|^2 \rightarrow 0 \;\;\;as\; n\rightarrow \infty.
 \end{align}
  By Lemma \ref{2.1.0}, it is implied that
 \begin{equation}\label{e}
   \displaystyle\lim_{n\rightarrow\infty}\|y_n-z_n\|= 0 .
 \end{equation}
  Since $\{f(x_n)\}$ and  $\{z_n\}$ are bounded. Now, setting $r_{1}=sup\{\| f(x_n)\| , \| z_n\| \}$, by
  Lemma \ref{2.4} there exists a continuous strictly increasing and convex function $ g_{1}:[0,2r_{1}]\longrightarrow [0,\infty ]$ with $g_{1}(0)=0$. From
  \eqref{three}, Lemmas \ref{2.1.2}, \ref{2.4} and the condition relatively nonexpansiveness of $f$, it is concluded for each $\hat{u}\in \Gamma$ that
\begin{align*}
      \phi(\hat{u}, x_{n+1})\leq& \phi(\hat{u}, J^{-1}(\alpha_{n,1}Jx_n +\alpha_{n,2}Jf(x_n)+\alpha_{n,3}Jz_n)\nonumber\\
      =&\| \hat{u}\|^2-2\langle \hat{u}, \alpha_{n,1}Jx_n +\alpha_{n,2}Jf(x_n)+\alpha_{n,3}Jz_n\rangle \nonumber\\
      &+\| \alpha_{n,1}Jx_n +\alpha_{n,2}Jf(x_n)+\alpha_{n,3}Jz_n\|^2\nonumber\\
      \leq& \| \hat{u}\|^2-2\alpha_{n,1}\langle \hat{u}, Jx_n\rangle-2\alpha_{n,2}\langle \hat{u}, Jf(x_n)\rangle-2\alpha_{n,3}\langle \hat{u}, Jz_n\rangle\nonumber \\
      &+\alpha_{n,1}\|x_n\|^2+\alpha_{n,2}\| f(x_n)\|^2+\alpha_{n,3}\| z_n\|^2\nonumber\\
      &-\alpha_{n,2}\alpha_{n,3}g_{1}(\| Jf(x_n)-Jz_n\|)\nonumber\\
       =& \alpha_{n,1}\phi(\hat{u},x_n)+\alpha_{n,2}\phi(\hat{u},f(x_n))+\alpha_{n,3}\phi(\hat{u}, z_n)\nonumber\\
        &-\alpha_{n,2}\alpha_{n,3}g_{1}(\| Jf(x_n)-Jz_n\|)\nonumber\\
      \leq & \alpha_{n,1}\phi(\hat{u},x_n)+\alpha_{n,2}\phi(\hat{u},x_n)+\alpha_{n,3}\phi(\hat{u}, x_n)\nonumber\\
       &-\alpha_{n,2}\alpha_{n,3}g_{1}(\| Jf(x_n)-Jz_n\|)\nonumber\\
      =& \phi(\hat{u}, x_n)-\alpha_{n,2}\alpha_{n,3}g_{1}(\| Jf(x_n)-Jz_n\|),
\end{align*}
therefore
\begin{equation*}
  \alpha_{n,2}\alpha_{n,3}g_{1}(\| Jf(x_n)-Jz_n\|)\leq \phi(\hat{u}, x_n)-\phi(\hat{u}, x_{n+1}).
\end{equation*}
Since $\liminf_{n\rightarrow\infty}\alpha_{n,2}\alpha_{n,3}>0$, we have
  \begin{equation}\label{g1}
      \lim_{n\rightarrow \infty} g_{1}(\| Jf(x_n)-Jz_n\|)= 0,
   \end{equation}
    because $\{\phi(\hat{u},x_n)\}$ is Cauchy and $\displaystyle\lim_{n\rightarrow \infty}\alpha_{n,2}\alpha_{n,3}>0$.
Since $g_{1}$ is a continuous function, so
    \begin{equation}\label{g11}
      g_{1}(\lim_{n\rightarrow \infty} \| Jf(x_n)-Jz_n\|)=
      \lim_{n\rightarrow \infty} g_{1}(\| Jf(x_n)-Jz_n\|)=0=g_{1}(0),
    \end{equation}
and also $g_{1}$  is strictly increasing, hence
 \begin{equation}\label{g111}
      \lim_{n\rightarrow \infty} \| Jf(x_n)-Jz_n\|)= 0.
   \end{equation}
On the other hand, since $J^{-1}$ is uniformly norm-to-norm continuous on bounded sets, we obtain that
     \begin{equation}\label{fxz}
      \lim_{n\rightarrow \infty}\| f(x_n)-z_n\|=\lim_{n\rightarrow\infty}\| J^{-1}(Jf(x_n))-J^{-1}( Jz_n)\|=0.
      \end{equation}

Next, from \eqref{phixy} and \eqref{fxz}, we have
\begin{equation}\label{fxnzn}
  \displaystyle\lim_{n\rightarrow\infty}\phi(z_n, f(x_n))=0.
\end{equation}
Similarly, from \eqref{phixy}, \eqref{c} and \eqref{e}, we obtain
\begin{equation}\label{pxy}
  \displaystyle\lim_{n\rightarrow\infty}\phi(z_n, x_n)=0.
\end{equation}
Moreover, from Lemma \ref{2.1.2}, the inequalities \eqref{fxnzn}, \eqref{pxy} and the convexity of $\|.\|^2$, it is
concluded that
   \begin{align*}
       \phi(z_n, x_{n+1})\leq&\phi(z_n,J^{-1}(\alpha_{n,1}Jx_n +\alpha_{n,2}Jf(x_n)+\alpha_{n,3}Jz_n))\\
       =&\| z_n\|^2-2\langle z_n,\alpha_{n,1}Jx_n +\alpha_{n,2}Jf(x_n)+\alpha_{n,3}Jz_n\rangle \\
       &+\| \alpha_{n,1}Jx_n +\alpha_{n,2}Jf(x_n)+\alpha_{n,3}Jz_n\|^2\\
       \leq& \| z_n\|^2-2\alpha_{n,1}\langle z_n,Jx_n\rangle-2\alpha_{n,2}\langle z_n,Jf(x_n)\rangle-2\alpha_{n,3}\langle z_n,Jz_n\rangle \\
       &+\alpha_{n,1}\|x_n\|^2+\alpha_{n,2}\| f(x_n)\|^2+\alpha_{n,3}\| z_n\|^2\\
       =&\alpha_{n,1}\phi(z_n, x_n)+\alpha_{n,2}\phi(z_n,f(x_n))+\alpha_{n,3}\phi(z_n,z_n)\\
       = &\alpha_{n,1}\phi(z_n, x_n)+\alpha_{n,2}\phi(z_n,f(x_n)) \rightarrow 0 \;\;\; as\;n\rightarrow \infty,
   \end{align*}
then using Lemma \ref{2.1.0}, we get
  \begin{equation}\label{f}
      \displaystyle\lim_{n\rightarrow\infty}\| z_n-x_{n+1}\| = 0.
   \end{equation}
It follows from \eqref{c}, \eqref{e} and \eqref{f} that
 \begin{equation}\label{xx}
      \| x_{n+1}-x_n\|\leq \| x_{n+1}-z_n\|+\| z_n-y_n\|+\|y_n-x_n\| \rightarrow 0 \;\;\;as\; n\rightarrow \infty.
   \end{equation}
So $\{x_n\}$ is a Cauchy sequence, thus  $\{x_n\}$ converges strongly to a point $q\in C$.
It follows from \eqref{c} and \eqref{e} that the sequences $\{y_n\}$ and $\{z_n\}$ are convergent to $q$.
Next, it will be shown that $q\in VI(C,A)$.
Let $B\subset E\times E^*$ be an operator defined as follows:
\begin{equation}\label{tlambdaa}
 B v =\left\{
\begin{array}{lr}
\lambda_n Av + N_{C}v,\qquad v \in C, \\
\emptyset,\qquad \qquad\qquad v \notin C.
\end{array} \right.
\end{equation}
Since $\lambda_nA$ is $\lambda_n\alpha$-inverse strongly monotone, it is followed that $\lambda_nA$ is  $\frac{1}{\lambda_n\alpha}$-Lipschitz continuous, hence
$\lambda_nA$ is hemicontinuous. Therefore, by Lemma \ref{2.60}  $B$ is maximal monotone and $B^{-1}(0)=VI(C,\lambda_nA)=VI(C, A)$.
Let $(\upsilon,w)\in G(B)$ with $w\in B\upsilon=\lambda_n A\upsilon+N_{C}(\upsilon)$.  Then $w-\lambda_n A\upsilon\in N_{C}(\upsilon)$, hence
   \begin{equation}\label{g}
     \langle\upsilon-y_n,w-\lambda_n A\upsilon\rangle\geq 0,
   \end{equation}
because $y_n\in C$. On the other hand by Lemma \ref{2.1.1}, it is concluded that
            $$\langle\upsilon-y_n, J(J^{-1}(Jx_n-\lambda_n Ax_n))-Jy_n\rangle \leq 0,$$
 so
   \begin{equation}\label{h}
     \langle\upsilon-y_n,\lambda_n Ax_n+Jy_n-Jx_n\rangle\geq 0.
   \end{equation}
From \eqref{g}, \eqref{h} and using the definition $A$, we get
   \begin{align}\label{k}
              \langle\upsilon-y_n,w\rangle&\nonumber\\
              \geq&\lambda_n\langle\upsilon-y_n,A\upsilon\rangle-\langle\upsilon-y_n,\lambda_n Ax_n+Jy_n-Jx_n\rangle\nonumber\\
             =&\lambda_n\langle\upsilon-y_n,A\upsilon-Ay_n\rangle + \lambda_n\langle\upsilon-y_n,Ay_n\rangle\nonumber\\
            & -\langle\upsilon-y_n,\lambda_n Ax_n+Jy_n-Jx_n\rangle\nonumber\\
             \geq&\lambda_n\langle\upsilon-y_n,Ay_n-Ax_n\rangle-\langle\upsilon-y_n,Jy_n-Jx_n\rangle\nonumber\\
            \geq& -\lambda_n\|\upsilon-y_n\|\| Ax_n-Ay_n\|-\|\upsilon-y_n\|\| Jx_n-Jy_n\|.
   \end{align}
Hence,  using uniformly norm-to-norm continuity of $J$ on bounded sets and \eqref{c},
 $\langle\upsilon-y_n, w\rangle \geq 0$ as $n\rightarrow \infty$, i.e. $\langle\upsilon-q,w\rangle \geq 0$. Therefore $\langle q-\upsilon,0-w\rangle \geq 0$, it
 is concluded from Lemma \ref{2.60} that $q\in B^{-1}(0)=VI(C,A)$, because $B$ is a maximal monotone operator.

Next, we show that $q\in F(f)$. From \eqref{c}, \eqref{e} and \eqref{fxz},  we have
   \begin{equation}\label{as}
         \| f(x_n)-x_n\| \leq \| f(x_n)-z_n\|+\| z_n-y_n\|+\|y_n-x_n\|\rightarrow 0\;\;\;as\;n\rightarrow \infty,
  \end{equation}
and since $x_n\rightharpoonup q$, then $q$ is an asymptotic fixed point of $f$.
Moreover, $\hat{F}(f)=F(f)$, because $f$ is a relatively nonexpansive mapping, hence $q\in F(f)$.
 Therefore, $\Pi_{VI(C,A)}of(q)=\Pi_{VI(C,A)}(q)=q$.
\end{proof}
  \begin{theorem}\label{3.2}
    Suppose that $\tilde{F}$ is a bifunction from $C\times C$ to $\mathbb{R}$ which satisfies the conditions $(A_{1})-(A_{4})$. Let f be a relatively nonexpansive
    self-mapping on $C$ and $\| Ax\|\leq \| Ax-Au\|$ for all $x\in C$ and $u\in \Omega:= VI(C, A)\cap GEP(\tilde{F}, A) \cap F(f)$. Let $x_0$ be an arbitrary point in $C$ and $\{x_n\}$ be a sequence
    generated by
  \begin{equation}\label{algo2}
    \left\{
    \begin{array}{lr}
       u_n\in C\;\; s.t\;\; \tilde{F}(u_n,y)+\langle Au_n,y-u_n\rangle+\frac{1}{r_n}\langle y-u_n,Ju_n-Jx_n\rangle\geq0, \\
       w_n=\Pi_{C}J^{-1}(Ju_n-\lambda_nAu_n),\\
       y_n=\Pi_CJ^{-1}(Jx_n-\lambda_nAx_n),\\
       C_n=\{v\in C : \phi (v, w_n)\leq \phi (v, x_n) \},\\
       z_n=\Pi_{C_n}J^{-1}(Jy_n-\lambda_nAy_n),\\
       x_{n+1}=\Pi_CJ^{-1}(\alpha_{n,1}Jx_n+\alpha_{n,2}Jf(x_n)+\alpha_{n,3}Jz_n+\alpha_{n,4}Jw_n).
    \end{array} \right.
   \end{equation}
where $\{\lambda_n\}\subseteq [0, 1]$ such that $\displaystyle\lim_{n\rightarrow \infty}\lambda_n=0$, and $\{r_n\}\subset [a,\infty)$ for some $a>0$. If
$\{\alpha_{n,i}\}\subset[0,1]$ for $i=1, 2, 3, 4$ such that $\sum_{i=1}^{4}\alpha_{n,i}=1$ and $\displaystyle\liminf_{n\to \infty}\alpha_{n,2}\alpha_{n,3}>0$ and $\displaystyle\liminf_{n\to \infty}\alpha_{n,2}\alpha_{n,4}>0$. Then the sequence $\{x_n\}$ generated by \eqref{algo2} converges strongly to $q=\Pi_{VI(C,A)\cap GEP(\tilde{F},A)}\circ f(q)$.
  \end{theorem}
  \begin{proof}
Clearly, by part $(i)$ of lemma \ref{2.7}, the sequence $\{u_n\}$ exists.   Now, it will be checked that $C_n$ is closed and convex for each $n \geq 1$. Obviously, by the definition of $C_n$, it is clear that $C_n$ is closed. Applying the definition of $\phi$, the inequality  $\phi (v, w_n)\leq \phi (v, x_n) $ is equivalent to
    \begin{equation}\label{convex}
     2\langle v, Jx_n- Jw_n\rangle\leq \|x_n\|^2 -\|w_n\|^2.
    \end{equation}
It is obvious from \eqref{convex} that $C_n$ is convex for all $n\geq 1$.

Now, it will be verified that $\{x_n\}$ is well defined. Suppose that $p\in \Omega$. By Lemma \ref{2.7}, it may be put $u_n= K_{r_n}x_n$. So, by the
condition (5) of Lemma \ref{2.7}, it is concluded that
   \begin{equation}\label{un<xn}
     \phi(p, u_n)= \phi(p,  K_{r_n}x_n)\leq \phi(p, x_n).
   \end{equation}
   Moreover,
   from Lemma \ref{2.1.2} and the inequality \eqref{vq}, it follows that
    \begin{align}\label{ineq41}
      \phi(p,  w_n)=& \phi(p,\Pi_{C}J^{-1}(Ju_n-\lambda_nAu_n))\nonumber \\
      \leq & \phi(p,J^{-1}(Ju_n-\lambda_nAu_n))\nonumber \\
      \leq & V(p,Ju_n-\lambda_nAu_n)\nonumber\\
      \leq &V(p,Ju_n)-2\langle J^{-1}(Ju_n-\lambda_nAu_n)-p,\lambda_nAu_n\rangle\nonumber\\
      =&\phi(p,u_n)-2\lambda_n\langle u_n-p,Au_n\rangle\nonumber\\&+2\langle J^{-1}(Ju_n-\lambda_nAu_n)-J^{-1}(Ju_n),-\lambda_nAu_n\rangle,
    \end{align}
since $A$ is an $\alpha-$inverse strongly monotone operator, it is proved that
    \begin{align}\label{ineq51}
      -2\lambda_n\langle u_n-p, & Au_n\rangle \nonumber\\
      = & -2\lambda_n\langle u_n-p,Au_n-Ap\rangle -2\lambda_n\langle u_n-p,Ap\rangle\nonumber \\
      \leq & -2\lambda_n\alpha\| Au_n-Ap\|^2.
     \end{align}
From Lemma \ref{2.2} and the condition $\| Ax\|\leq \| Ax-Ap\|$ for all $x\in C$, it is demonstrated that
    \begin{align}\label{ineq61}
      2\langle J^{-1}(Ju_n-\lambda_n&Au_n)- J^{-1}(Ju_n),  -\lambda_nAu_n\rangle\nonumber\\
      \leq & 2\|J^{-1}(Ju_n-\lambda_nAu_n)-J^{-1}(Ju_n)\|\| \lambda_nAu_n\|\nonumber \\
      = & \frac{4\lambda_n^2}{c^2}\| Au_n\|^2\nonumber\\
      \leq& \frac{4\lambda_n^2}{c^2}\| Au_n-Ap\|^2.
    \end{align}
By substituting \eqref{ineq51} and \eqref{ineq61} in \eqref{ineq41} and the assumption $0<\lambda_n < \frac{c^2\alpha}{2}$, it is implied that
    \begin{equation}\label{pwpu}
       \phi(p,  w_n)
      \leq \phi(p, u_n)+2\lambda_n(\frac{2}{c^2}\lambda_n-\alpha)\|  Au_n-Ap\|^2\leq \phi(p, u_n).
      \end{equation}
From \eqref{un<xn} and \eqref{pwpu}, it is evident that
     \begin{equation}\label{wn<xn}
       \phi(p,  w_n)\leq \phi(p, x_n).
     \end{equation}
Then $p\in C_n$ and hence $\{x_n\}$ is well defined.

 Let $\Omega\neq\emptyset$ and $\hat{u}\in \Omega$. From Lemma \ref{2.1.2}, the convexity of $\| . \|^2$ and the relatively nonexpansiveness of $f$, it follows
     \begin{align*}
        \phi(\hat{u},   x_{n+1}&) \\
       \leq &\phi(\hat{u}, J^{-1}(\alpha_{n,1}Jx_n+\alpha_{n,2}Jf(x_n)+\alpha_{n,3}Jz_n+\alpha_{n,4}Jw_n))\\
       = & \| \hat{u}\|^2-2\langle \hat{u},\alpha_{n,1}Jx_n+\alpha_{n,2}Jf(x_n)+\alpha_{n,3}Jz_n+\alpha_{n,4}Jw_n\rangle\\
        &+\| \alpha_{n,1}Jx_n+\alpha_{n,2}Jf(x_n)+\alpha_{n,3}Jz_n+\alpha_{n,4}Jw_n\|^2 \\
       \leq &  \| \hat{u}\|^2-2\alpha_{n,1}\langle\hat{u}, Jx_n\rangle-2\alpha_{n,2}\langle \hat{u},Jf(x_n)\rangle
        -2\alpha_{n,3}\langle \hat{u}, Jz_n\rangle-2\alpha_{n,4}\langle \hat{u}, Jw_n\rangle\\
        &+\alpha_{n,1}\|x_n\|^2+\alpha_{n,2} \| f(x_n)\|^2+\alpha_{n,3} \| z_n\|^2+\alpha_{n,4}\| w_n\|^2\\
       = &\alpha_{n,1}\phi(\hat{u}, x_n)+\alpha_{n,2} \phi(\hat{u}, f(x_n))+\alpha_{n,3}\phi(\hat{u}, z_n)+\alpha_{n,4}\phi(\hat{u}, w_n)\\
       \leq &\alpha_{n,1}\phi(\hat{u}, x_n)+\alpha_{n,2} \phi(\hat{u}, x_n)+\alpha_{n,3}\phi(\hat{u}, z_n)+\alpha_{n,4}\phi(\hat{u}, w_n)\\
       =&(\alpha_{n,1}+\alpha_{n,2}) \phi(\hat{u}, x_n)+\alpha_{n,3}\phi(\hat{u}, z_n)+\alpha_{n,4}\phi(\hat{u}, w_n).
      \end{align*}
Similarly, using Lemma \ref{2.1.2}, the inequality \eqref{three} holds for the algorithm \eqref{algo2}, too. Hence, from \eqref{three}  and \eqref{wn<xn}, it is implied that
      \begin{equation}\label{xn,xn+1}
        \phi(\hat{u}, x_{n+1})\leq \phi(\hat{u}, x_n).
      \end{equation}
It is concluded that $\{\phi(\hat{u}, x_n)\}$ is decreasing, so from the boundedness of the sequence $\{\phi(\hat{u}, x_n)\}$,  $\displaystyle\lim_{n\rightarrow \infty }\phi(\hat{u}, x_n)$ exists. Also from \eqref{phi1}, $\{x_n\}$ is bounded and hence from \eqref{un<xn} and the relatively nonexpansiveness of $f$, $\{u_n\}$ and $\{f(x_n)\}$ are bounded.
Similarly, using Lemma \ref{2.1.2}, the inequalities \eqref{c} and \eqref{e} hold for the algorithm \eqref{algo2}.
Hence, it is concluded from  \eqref{c} and \eqref{e} that the sequences $\{y_n\}$ and $\{z_n\}$ are bounded. Now, let $r_1=\sup\{\|z_n\|, \|f(x_n)\|\}$, by Lemma \ref{2.4}, there exists a continuous strictly increasing and convex function $ g_{1}:[0,2r_{1}]\longrightarrow [0,\infty )$ with $g_{1}(0)=0$. We get
 \begin{align*}
        \phi(\hat{u},   x_{n+1}&) \\
       \leq &\phi(\hat{u}, J^{-1}(\alpha_{n,1}Jx_n+\alpha_{n,2}Jf(x_n)+\alpha_{n,3}Jz_n+\alpha_{n,4}Jw_n))\\
       = & \| \hat{u}\|^2-2\langle \hat{u},\alpha_{n,1}Jx_n+\alpha_{n,2}Jf(x_n)+\alpha_{n,3}Jz_n+\alpha_{n,4}Jw_n\rangle\\
        &+\| \alpha_{n,1}Jx_n+\alpha_{n,2}Jf(x_n)+\alpha_{n,3}Jz_n+\alpha_{n,4}Jw_n\|^2 \\
       \leq &  \| \hat{u}\|^2-2\alpha_{n,1}\langle\hat{u}, Jx_n\rangle-2\alpha_{n,2}\langle \hat{u},Jf(x_n)\rangle
        -2\alpha_{n,3}\langle \hat{u}, Jz_n\rangle\\
        &-2\alpha_{n,4}\langle \hat{u}, Jw_n\rangle+\alpha_{n,1}\|x_n\|^2+\alpha_{n,2} \| f(x_n)\|^2+\alpha_{n,3} \| z_n\|^2\\
        &+\alpha_{n,4}\| w_n\|^2-\alpha_{n,2}\alpha_{n,3}g_1(\|Jf(x_n)-Jz_n\|)\\
       = &\alpha_{n,1}\phi(\hat{u}, x_n)+\alpha_{n,2} \phi(\hat{u}, f(x_n))+\alpha_{n,3}\phi(\hat{u}, z_n)+\alpha_{n,4}\phi(\hat{u}, w_n)\\
       &-\alpha_{n,2}\alpha_{n,3}g_1(\|Jf(x_n)-Jz_n\|)\\
       \leq &\alpha_{n,1}\phi(\hat{u}, x_n)+\alpha_{n,2} \phi(\hat{u}, x_n)+\alpha_{n,3}\phi(\hat{u}, z_n)+\alpha_{n,4}\phi(\hat{u}, w_n)\\
       &-\alpha_{n,2}\alpha_{n,3}g_1(\|Jf(x_n)-Jz_n\|)\\
       =&(\alpha_{n,1}+\alpha_{n,2}) \phi(\hat{u}, x_n)+\alpha_{n,3}\phi(\hat{u}, z_n)+\alpha_{n,4}\phi(\hat{u}, w_n)\\
       &-\alpha_{n,2}\alpha_{n,3}g_1(\|Jf(x_n)-Jz_n\|).
      \end{align*}
      Now from \eqref{three} and \eqref{wn<xn}, we have
      \begin{equation}\label{ara}
        \phi(\hat{u},   x_{n+1})\leq \phi(\hat{u}, x_n)-\alpha_{n,2}\alpha_{n,3}g_1(\|Jf(x_n)-Jz_n\|).
      \end{equation}
      So
      \begin{equation*}
        \alpha_{n,2}\alpha_{n,3}g_1(\|Jf(x_n)-Jz_n\|)\leq \phi(\hat{u}, x_n)-\phi(\hat{u},   x_{n+1}).
      \end{equation*}
      Since $\liminf_{n\rightarrow\infty}\alpha_{n,2}\alpha_{n,3}>0$, using the method as in the proof of Theorem \ref{jmen}, we conclude that the inequality \eqref{fxnzn} and \eqref{pxy} hold.

       By Lemma \ref{2.1.2} and convexity of $\| .\|^2$, it is obtained that
      \begin{align*}
        \phi(z_n, x_{n+1}&)\nonumber\\
         \leq & \phi(z_n,J^{-1}(\alpha_{n,1}Jx_n+\alpha_{n,2}Jf(x_n)+\alpha_{n,3}Jz_n+\alpha_{n,4}Jw_n))\nonumber \\
         =& \| z_n\|^2-2\langle z_n, \alpha_{n,1}Jx_n+\alpha_{n,2}Jf(x_n)+\alpha_{n,3}Jz_n+\alpha_{n,4}Jw_n\rangle\nonumber \\
          &+ \| \alpha_{n,1}Jx_n+\alpha_{n,2}Jf(x_n)+\alpha_{n,3}Jz_n+\alpha_{n,4}Jw_n\|^2\nonumber \\
         \leq &\| z_n\|^2-2\alpha_{n,1}\langle z_n, Jx_n\rangle-2\alpha_{n,2}\langle z_n, Jf(x_n)\rangle-2\alpha_{n,3} \langle z_n,Jz_n\rangle\nonumber\\
         &-2\alpha_{n,4} \langle z_n,Jw_n\rangle+\alpha_{n,1}\|x_n\|^2+\alpha_{n,2}\| f(x_n)\|^2+\alpha_{n,3}\|z_n\|^2\nonumber\\
         &+\alpha_{n,4}\|w_n\|^2\nonumber\\
         =&\alpha_{n,1}\phi(z_n, x_n)+\alpha_{n,2}\phi(z_n, f(x_n))+\alpha_{n,3}\phi(z_n,z_n)+\alpha_{n,4}\phi(z_n,w_n)\nonumber\\
         \leq&(\alpha_{n,1}+\alpha_{n,4})\phi(z_n, x_n)+\alpha_{n,2}\phi(z_n, f(x_n)),
      \end{align*}
because $z_n\in C_n$.
Using \eqref{fxnzn}, \eqref{pxy} and taking the limit in the above as $n\rightarrow\infty$, it is deduced that
\begin{equation*}
        \phi(z_n,x_{n+1})\rightarrow 0.
      \end{equation*}
Then, from Lemma \ref{2.1.0}, we have
\begin{equation*}
  \displaystyle\lim_{n\rightarrow\infty}\|x_{n+1}-z_n\|=0,
\end{equation*}
therefore, it follows from \eqref{c}, \eqref{e} that
      \begin{align*}
        \| x_{n+1}-x_n\|\leq \|  x_{n+1}- z_n\|+\| z_n-y_n\|+\|y_n-x_n\|\rightarrow 0\;\;as\;\; n\rightarrow \infty,
      \end{align*}
hence,  $\{x_n\}$ is a cauchy sequence. Thus, $\{x_n\}$ converges strongly to a point $q\in C$. Obviously, the relations \eqref{tlambdaa}, \eqref{g},
\eqref{h} and \eqref{k} are valid for the algorithm \eqref{algo2}.  Hence, as in the proof of Theorem \ref{jmen}, it is understood that $q \in V
I(C, A)$.

Now, it will be proved that $q\in GEP(\tilde{F}, A)$.
From \eqref{pxy} and the fact that $z_n\in C_n$, it is induced that $\phi(z_n, w_n)\rightarrow0$ as $n\rightarrow \infty$. Therefore, by  Lemma \ref{2.1.0}, we
have
   \begin{equation}\label{ynwn}
     \displaystyle\lim_{n\rightarrow\infty}\|z_n - w_n\|=0.
   \end{equation}
From \eqref{c}, \eqref{e} and \eqref{ynwn}, it is evident that
     \begin{equation}\label{xnwn}
       \displaystyle\lim_{n\rightarrow\infty}\|x_n-w_n\|=0.
     \end{equation}
Assume that $r_{2}=\sup\{\| u_n\| ,\| x_n\|\}$. From Lemma \ref{2.5}, there exists  a continuous, convex and strictly increasing function
$g_{2}:[0,2r_{2}]\longrightarrow [0,\infty)$ such that $g_{2}(0)=0$ and
    \begin{align}\label{g<phi}
      g_{2}(\| u_n-x_n\|)\leq\phi( u_n,x_n).
    \end{align}
Since $ u_n=K_{r_n}(x_n)$ and by using \eqref{pwpu}, \eqref{g<phi} and condition $(5)$ of Lemma \ref{2.7}, it is implied that
    \begin{align*}
      g_{2}(\| u_n-x_n\|)\leq&\phi( u_n,x_n) \\
      \leq & \phi( u,x_n)-\phi( u,u_n) \\
      \leq & \phi( u,x_n)-\phi( u,w_n)  \\
      = & \| u\|^2-2\langle u,Jx_n\rangle+\| x_n\|^2-\| u\|^2+2\langle u,Jw_n\rangle-\| w_n\|^2\\
      = & \| x_n\|^2-\| w_n\|^2+2\langle u,Jw_n-Jx_n\rangle \\
      \leq & \| x_n\|^2-\| w_n\|^2+2\| u\|\| Jw_n-Jx_n\| \\
      \leq & (\| x_n-w_n\|+\| w_n\|)^2-\| w_n\|^2+2\| u\|\| Jw_n-Jx_n\| \\
      \leq & \| x_n-w_n\|^2+2\| w_n\|\| x_n-w_n\|+2\| u\|\| Jw_n-Jx_n\|,
      \end{align*}
from \eqref{xnwn} and the condition uniformly norm-to-norm continuity of $J$ on bounded sets, we have $\displaystyle\lim_{n\rightarrow\infty} g_{2}(\| u_n-x_n\|)=0$. Then it is followed from the conditions that $g_{2}$ is a strictly increasing and continuous function that $\| u_n-x_n\|\rightarrow0$ as $n\rightarrow\infty$. Then
      \begin{equation}\label{jujx}
        \lim_{n\rightarrow\infty}\| Ju_n-Jx_n\|\rightarrow0.
      \end{equation}
Since $u_n=K_{r_n}x_n$, we concluded that
       \begin{equation}\label{3}
         \tilde{F}(u_n,y)+\langle Au_n,y-u_n\rangle+\frac{1}{r_n}\langle y-u_n,Ju_n-Jx_n\rangle\geq0,
       \end{equation}
for all $y\in C$. From the condition $(A_{2})$, we have
       \begin{equation}\label{4}
         \tilde{F}(y,u_n)\leq - \tilde{F}(u_n,y)\;\; \text{for}\;\text{all}\; y\in C.
       \end{equation}
From \eqref{3} and \eqref{4}, it is implied that
        \begin{align*}
          \tilde{F}(y,u_n)\leq - \tilde{F}(u_n,y)\leq \langle Au_n,y-u_n\rangle+\frac{1}{r_n}\langle y-u_n,Ju_n-Jx_n\rangle,
        \end{align*}
for all $y\in C$. Letting $n\rightarrow \infty$, using condition $(A_{4})$ and by \eqref{jujx}, it can be concluded that
        \begin{equation}\label{FAq}
          \tilde{F}(y,q)\leq \langle Aq,y-q\rangle\;\;\;\text{for}\;\text{all}\; y\in C.
        \end{equation}
Put $y_{\lambda}=\lambda y+(1-\lambda)q$ for all $y\in C$ and $\lambda\in (0,1)$. Now from the conditions $(A_{1})$, $(A_4)$, the inequality \eqref{FAq}, the monotonicity of $A$ and the convexity of
$\tilde{F}$, we have
        \begin{align*}
          0=&\tilde{F}(y_{\lambda},y_{\lambda})+\langle Ay_{\lambda},y_{\lambda}-y_{\lambda}\rangle  \\
          \leq & \lambda \tilde{F}(y_{\lambda},y)+(1-\lambda)\tilde{F}(y_{\lambda},q)+\langle Ay_{\lambda},\lambda y+(1-\lambda)q-y_{\lambda}\rangle \\
          = & \lambda \tilde{F}(y_{\lambda},y)+(1-\lambda)\tilde{F}(y_{\lambda},q)+\lambda\langle Ay_{\lambda},y-y_{\lambda}\rangle
          +(1-\lambda)\langle Ay_{\lambda},q-y_{\lambda}\rangle\\
          =& \lambda \tilde{F}(y_{\lambda},y)+(1-\lambda)\tilde{F}(y_{\lambda},q)+\lambda\langle Ay_{\lambda},y-y_{\lambda}\rangle
          +(1-\lambda)\langle Ay_{\lambda}-Aq,q-y_{\lambda}\rangle \\
          &+(1-\lambda)\langle Aq, q-y_{\lambda}\rangle\\
          \leq &\lambda \tilde{F}(y_{\lambda},y)+\lambda\langle Ay_{\lambda},y-y_{\lambda}\rangle,
        \end{align*}
for all $y\in C$.
So $0\leq \tilde{F}(y_{\lambda},y)+\langle Ay_{\lambda},y-y_{\lambda}\rangle$. Now by taking limit as $\lambda\rightarrow 0$ and by using the condition $(A_{3})$, it is followed that $0\leq\tilde{F}(q,y)+\langle Aq,y-q\rangle$ for all $y\in C$. Therefore
 $q\in GEP(\tilde{F}, A)$.

 Now, we show that $q \in F(f)$. Let $r_3=\sup\{\| w_{n}\| ,\| f(x_{n})\|\}$,
hence, in a similar way with \eqref{ara}, there exists a continuous, convex and strictly increasing function $g_3:[0,2r_3]\longrightarrow [0,\infty)$ whit $g_3(0)=0$, such that
     \begin{equation*}
       \phi(\hat{u},   x_{n+1})\leq\phi(\hat{u}, x_n) -\alpha_{n,2}\alpha_{n,4}g_3(\|Jf(x_n)-Jw_n\|),
     \end{equation*}
hence
      \begin{equation*}
        \alpha_{n,2}\alpha_{n,4}g_3(\| Jf(x_{n})-Jw_{n}\|)\leq \phi(\hat{u}, x_{n+1})-\phi(\hat{u}, x_{n}).
       \end{equation*}
Taking the limit as $n\longrightarrow \infty$ and using our assumptions, we obtain
       \begin{align*}
         \lim_{n\rightarrow \infty}g_3(\| Jf(x_{n})-Jw_{n}\|)=0,
       \end{align*}
since $g_1$ is a continuous function, it is easy to see that
      \begin{align}
         \lim_{n\rightarrow \infty}\| Jf(x_{n})-Jw_{n}\|=0.
      \end{align}
Therefore
      \begin{equation}\label{fx-z}
         \lim_{n\rightarrow \infty}\| f(x_{n})-w_{n}\|=\lim_{n\rightarrow \infty}\|J^{-1}(Jf(x_{n}))-J^{-1}(Jw_{n})\|=0,
      \end{equation}
because $J^{-1}$ is uniformly norm-to-norm continuous on bounded sets. From \eqref{xnwn} and \eqref{fx-z}, it is concluded that
      \begin{equation*}
        \|f(x_n) -x_n\| \leq \|f(x_n)-w_n\| + \|w_n - x_n\| \rightarrow 0 \;\;\; as \;n\rightarrow \infty,
      \end{equation*}
and since $x_n\rightharpoonup q$, then $q\in \hat{F(f)}= F(f)$. Hence $\{x_n\}$ is strongly convergent to a point $q\in \Omega$, and also we have $q=\Pi_{VI(C,A)\cap  GEP(\tilde{F},A)}\circ f(q)$.
\end{proof}
  \section{Numerical example}
Now, some examples are given to illustrate Theorem \ref{3.2}. Then the behaviors of the sequences $\{x_n\}, \{y_n\}, \{z_n\} $ and $\{w_n\}$ are investigated which were generated by the algorithm \eqref{algo2}.
    \begin{example}\label{ex}
      Let $E=\mathbb{R}$, $C=[-5,5]$, $A=I$, $\lambda_n=\frac{1}{n}$, $c=1$, $\alpha=1$ and $f$ be a self-mapping on $C$ defined by  $f(x)=\frac{x}{3}$ for all $x \in C$. Consider the function $\tilde{F}: C\times C\rightarrow \mathbb{R}$ defined by
     \begin{equation*}
       \tilde{F}(u,y):=16y^2+9uy-25u^2,
     \end{equation*}
for all $u$, $y \in C$.
We see that $f$ satisfies in the conditions (A1) - (A4) as follows:\\
(A1) $\tilde{F}(u,u)=16u^2+9u^2-25u^2=0$ for all $u\in [-5,5]$,\\
(A2) $\tilde{F}$ is monotone, because $\tilde{F}(u,y)+\tilde{F}(y,u)=-9(u-y)^2\leq 0$ for all $y,u\in [-5,5]$,\\
(A3) for each $u, y, z\in [-5,5],$
      \begin{align*}
        \lim_{\lambda \to 0}\tilde{F}(\lambda z+(1&-\lambda)u,y)\\
        =&\lim_{\lambda \to 0}(16y^2+9(\lambda z+(1-\lambda)u)y-25(\lambda z+(1-\lambda)u)^2) \\
        =& 16y^2+9uy-25u^2\\
        =&\tilde{F}(u,y).
      \end{align*}
(A4) Obviously, for each $u\in [-5, 5]$, $y\rightarrow (16y^2+9uy-25u^2)$ is convex and lower semicontinuous.\\
Let $u \in K_rx$, hence, it is concluded from Lemma \ref{2.7} that
\begin{align*}
   \tilde{F}(u, y)+ \langle Au, y-u\rangle+\frac{1}{r}\langle y-u, Ju-Jx\rangle \geq 0,
\end{align*}
for all $y \in [-5, 5]$ and $r > 0$,
i.e.,
    \begin{align*}
        0\leq 16ry^2+9ruy-25ru^2+&ruy-ru^2+uy-u^2+ux-xy\\
         = &16ry^2+(10ru+u-x)y-26ru^2-u^2+ux.
    \end{align*}
Let $a=16r$, $b = 10ru+u-x$ and $c = -26ru^2-u^2+ux$. Then, it is implied that $\triangle = b^2 -4ac \leq 0$, i.e.,
\begin{align*}
  0\geq (10ru+u-x)^2-64r(-26&ru^2-u^2+ux)\\
  = &1764r^2u^2+84ru^2+u^2-84rux-2ux+x^2\\
  =&((42r+1)u-x)^2.
\end{align*}
It follows that $u=\frac{x}{42r+1}$. It is concluded from Lemma \ref{2.7} that $K_r$ is single valued. Hence, $K_rx=\frac{x}{42r+1}$. Now by applying in theorem \ref{3.2}, it is implied that $u_n=\frac{x_n}{42r_n+1}$ where $\{x_n\}$ is a sequence generated by the algorithm \eqref{algo2}. Since
$F(K_{r_n})=\{0\}$, from condition $(3)$ of Lemma \ref{2.7}, we have $GEP(\tilde{F}, I)=\{0\}$.\\
Obviously,  $F(f)=\{ 0 \}$ and $\phi(0, f(x))\leq \phi(0, x)$, for all $x\in C$.
Now, let $x_n \rightharpoonup q$ and also $\lim_{n \to \infty} (f(x_n)-x_n)=0$, hence $q=0$ and $\hat{F(f)}=\{0\}=F(f)$. Therefore, $f$ is a relatively
nonexpansive mapping. Moreover, it is obvious that $0\in VI(C, I)$. Therefore, $0=\Pi_{\{0\}}o f(0) = \Pi_{VI(C, I)\cap GEP(\tilde{F}, I)}o f(0)$.

Next, assume that $\alpha_{n,1}=\frac{1}{4}+\frac{1}{4n}, \alpha_{n,2}=\frac{1}{4}-\frac{1}{6n}, \alpha_{n,3}=\frac{1}{4}+\frac{1}{12n}, \alpha_{n,4}=\frac{1}{4}-\frac{1}{6n}$, for all $n \in \mathbb{N}$ and $u_0=0$, so clearly
$\alpha_n$, $\beta_n$ and $\gamma_n$ satisfy in the conditions of Theorem \ref{3.2}.
Since $x_n\in C$, we have
    \begin{equation*}\label{algo3}
      \left\{
       \begin{array}{lr}
       w_n=\Pi_CJ^{-1}(u_n-\frac{1}{n}u_n)=\frac{n-1}{n}u_n=\frac{n-1}{2n}x_n, \\
       y_n=\Pi_CJ^{-1}(x_n-\frac{1}{n}x_n)=\Pi_C\frac{n-1}{n}x_n=\frac{n-1}{n}x_n,  \\
       C_n=\{v\in C:|v-w_n |\leq |v - x_n |\},\\
       z_n=\Pi_{C_n}J^{-1}(y_n-\frac{1}{n}y_n)= \frac{n-1}{n}y_n=(\frac{n-1}{n})^2x_n,\\
       x_{n+1}=\Pi_CJ^{-1}((\frac{1}{4}+\frac{1}{4n})x_n+(\frac{1}{4}-\frac{1}{6n})\frac{1}{3}x_n+(\frac{1}{4}+\frac{1}{12n})(\frac{n-1}{n})^2
       x_n\\
       \hspace{1.1cm}+(\frac{1}{4}-\frac{1}{6n})\frac{n-1}{2n}x_n).
     \end{array} \right.
    \end{equation*}
 \end{example}
See the table \ref{tableexample1} and Figure \ref{pp1} with the initial point $x_1=5$ of the sequence $\{x_n\}$.

    %


%
\section*{Declarations}
Not applicable.
\section*{Funding}
No funding is applicable to this article.
\section*{Conflict of interest}
The authors declare that they have no conflict of interest.
\section*{Authors' contributions}
The two authors equally contributed, read, and approved the final manuscript.
\section*{Competing interests}
The authors declare that they have no competing interests.
\section*{Acknowledgements}
The authors would like to thank the referees for their esteemed comments and suggestions.




\end{document}